\documentclass[12pt]{article}
\usepackage{amsmath}
\usepackage{amsthm}
\usepackage{amssymb}
\usepackage{graphics}
\usepackage{enumerate}   
\newcommand{\be}{\begin{equation}}
\newcommand{\ee}{\end{equation}}
\newcommand{\ben}{\begin{eqnarray*}}
\newcommand{\een}{\end{eqnarray*}}
\newtheorem{examp}{\sc Example}
\newtheorem{remk}{\sc Remark}
\newtheorem{corol}{\sc Corollary}
\newtheorem{lemma}{\sc Lemma}
\newtheorem{theorem}{\sc Theorem}
\newtheorem{defn}{\sc Definition}
\newcommand{\bt}{\begin{theorem}}
\newcommand{\et}{\end{theorem}}
\newcommand{\bl}{\begin{lemma}}
\newcommand{\el}{\end{lemma}}
\newcommand{\bed}{\begin{defn}}
\newcommand{\eed}{\end{defn}}
\newcommand{\brem}{\begin{remk}}
\newcommand{\erem}{\end{remk}}
\newcommand{\bex}{\begin{examp}}
\newcommand{\eex}{\end{examp}}
\newcommand{\bcl}{\begin{corol}}
\newcommand{\ecl}{\end{corol}}

\newcommand{\newsection}[1]{\setcounter{equation}{0} \setcounter{theorem}{0}
\setcounter{lemma}{0} \setcounter{defn}{0} \setcounter{remk}{0}
\setcounter{corol}{0} \setcounter{examp}{0}  \section{#1}}

\newcommand{\NI}{\noindent}

\newcommand{\raro}{\rightarrow}

\newcommand{\vsp}{\vskip 0.5em}

\theoremstyle{definition}
\theoremstyle{remark}

\numberwithin{equation}{section}
\numberwithin{theorem}{section}
\numberwithin{lemma}{section}

\begin{document}
	\title{\large {\bf{\sc Iterative Descent Method for Generalized Leontief Model}}}
	
	\author{R. Jana$^{a, 1}$, A. K. Das$^{b,2}$, Vishnu Narayan Mishra$^{c,3}$\\
		\emph{\small $^{a}$Jadavpur University, Kolkata}\\
		\emph{\small $^{b}$Indian Statistical Institute, Kolkata}\\
		\emph{\small $^{c}$Indira Gandhi National Tribal University, Amarkantak}\\
		\emph{\small $^1$Email: rwitamjanaju@gmail.com}\\
		\emph{\small $^2$Email: akdas@isical.ac.in}\\
		\emph{\small $^3$Email: vishnunarayanmishra@gmail.com}}
	\date{}
	
	\maketitle
	
	\begin{abstract}
		\noindent In this paper we consider generalized Leontief model. We show that under certain condition the generalized Leontief model is solvable by iterative descent method based on infeasible interior point algorithm. We prove the convergence of the method from strictly positive starting point. A numerical example is presented to demonstrate the performance of the algorithm. 
		
		\vsp
		\NI{\bf Keywords:} Leontief model, generalized Leontief model, linear complementarity problem, vertical linear complementarity problem, infeasible interior point algorithm.
	\end{abstract}
	
	\footnotetext[3]{Corresponding author}
	
	\newsection{Introduction}  
	The purpose of Leontief model \cite{leontief1} is to find the interrelationship among goods and services for different sectors of the economy. Leontief model considers production of items within some industries where number of industries and number of products are equal. In other words the model indicates a balance between demand and supply. The model is very useful to analyze the national economy of any sector as each of the industries uses input from itself and other industries to produce a particular product. 
	
	Leontief model is classified as open model and closed model (see \cite{berman}). Open model deals with finding the production level based on external demand whereas the closed model deals only with internal demand. The input-output model has wide applications in the area of regional economics \cite{regionaleco}, international trade \cite{Barker}, multi facility inventory systems \cite{veinott}. The Leontief model describes a facilitated view of an economical situation. The target of this model is to state the exact level of production for each of various types of services or goods. Suppose $b_j$ denotes the units available or required at industry $j,$ and $a_{jk}$ be the technical coefficients representing units of output of sector $j$ required per unit output of sector $k.$ The net output $b_j$ is normally called the final demand of the $j$th good. Suppose the basic input-output equations are given as
	\be \label{cl}
	x_j = \sum_{k = 1}^{n} a_{jk}x_k + b_j
	\ee
	then we can find $x_j$ which is the output from industry $j.$ 
	
	There is an assumption that each industry or sector produces only one output. In case if an industry produces more than one output, then the analysis is done by aggregation. Several variants of the input-output model are available in the literature \cite{variants}. Leontief model considers single technology. Ebiefung et al. \cite{ebiefung} introduced generalized Leontief model by considering multiple technologies. In this paper we consider an infeasible interior point method in line with Kojima et al. \cite{kojima} and show that generalized Leontief model can be solved using this method by controlling step lengths. 
	
	The paper is organized as follows. Section 2 presents notations and some basic results related to the Leontief model and generalized Leontief model. In section 3, we propose an infeasible interior point method in line with Kojima et al. \cite{kojima} to solve generalized Leontief model. We prove the convergence of the algorithm. A numerical example is given in section 4 to illustrate the performance of the proposed algorithm. 
	
	\section{Preliminaries}
	\noindent We consider matrices and vectors with real entries. Any vector $z \in R^n$ is a column vector, $z^{T}$ denotes the transpose of $z$. $z_k$ denotes the $k$th component of the vector $z$. $e$ denotes the vector of all $1$. $\|z\|$ denotes the euclidean norm of the vector $z$. For any matrix $A \in R^{n \times n}$, $A^{T}$ denotes its transpose. $R^{n}$ denotes real $n$-dimensional space. $R^{n}_{+}$ and $R^{n}_{++}$ donote the nonnegative and positive orthant in $R^{n}$ respectively. 
	
	\subsection{Leontief Model}
	Suppose there are $n$ industries and let $b_j$ be the demand for output $j$ where $j = 1, 2, \cdots, n$. $b_j \geq 0$ implies the requirement at the end of the period and $b_j \leq 0$ represents the number of units available at the beginning of the period. Generalizing the Leontief model (\ref{cl}), the requirement that each output at industry $j$ can be written as,
	\begin{equation} \label{leontief}
	x_j \geq b_j + \sum_{k = 1}^{n} a_{jk} x_k.
	\end{equation}
	Again Equation \ref{leontief} can be rewritten as 
	\be \label{f1}
	x \geq 0
	\ee
	\be \label{f2}
	-b + (I-A)x \geq 0
	\ee
	\be \label{complementary}
	x^T[-b + (I-A)x] = 0
	\ee
	where $A = (a_{jk})$ and $I$ is the identity matrix of order $n$. Hence Equation \ref{f1}, Equation \ref{f2} and Equation \ref{complementary} can be written as LCP$(q, M)$ by taking $q = -b$ and $M = (I-A).$ The complementary condition implies either $x_j > 0,$ then $(-b + Mx)_j = 0$ i.e. production of output $j$ at industry $j$ meets demand exactly or $x_j = 0$ then from Equation \ref*{leontief}, we get $-b_j > \sum_{k = 1}^{n} a_{jk} x_k$ i.e. no production is necessary at industry $j$.
	
	Now we consider the following result on Leontief model which will be used in the subsequent section.
	\begin{theorem} \cite{berman} \label{berman}
		In an open Leontief model with input matrix $T$ and let $A = I - T.$ Then the following statements are equivalent:
		\begin{enumerate}
			\item The model is feasible.
			\item The model is profitable.
			\item $A$ is non singular $M$ matrix.
		\end{enumerate} 
	\end{theorem}
	\begin{proof}
		If the model is feasible then by choosing demand vector $d > 0,$ it follows that there exists a $x \geq 0$ with $Ax = d > 0.$ This condition characterizes the property of non singular $M$ matrices. Conversely, suppose $A$ is non singular $M$ matrix, then $A^{-1} \geq 0.$ Thus $Ax = d$ has a nonnegative solution $x = A^{-1}d$ for each $d \geq 0.$ The equivalence of the model is profitable and $A$ is nonsingular $M$ matrix follows from the fact that $A$ is non singular $M$ matrix if and only if $A^T$ is non singular $M$ matrix.
	\end{proof}
	
	There are several solution procedure of Leontief model developed in the recent times. Dantzig \cite{Dantzig} presented a method to solve this model using substitution technique.
	
	\subsection{Vertical Linear Complementarity Problem}
	Consider a rectangular matrix $N$ of order $m \times k$ with $m \geq k.$ Suppose $N$ is partitioned row-wise into $k$ blocks in the form
	\begin{center}
		$N$ =  $\left[\begin{array}{r}
		N^1 \\
		N^2 \\
		\vdots\\
		N^k
		\end{array}\right]$
	\end{center}
	where each $N^j = (n^{j}_{rs}) \in R^{m_j \times k}$ with $\sum_{j = 1}^{k} m_j = m.$ Then $N$ is called vertical block matrix of type $(m_1, m_2, \cdots, m_k).$ If $m_j = 1, \forall j = 1, 2, \cdots, k,$ then $N$ is a square matrix. The concept of vertical block was introduced by Cottle and Dantzig \cite{cd} in connection with the generalization of linear complementarity problem. Cottle-Dantzig's generalization involves a system $w - Nz = q,\ w \geq 0, \ z \geq 0,$ where $N \in R^{m \times k},$ $m \geq k$ and the variable $w_1, w_2, \cdots, w_m$ are partitioned into $k$ nonempty sets $\mathcal{S}_{j}, j = 1, 2, \cdots, k.$ Let $\mathcal{T}_{j} = \mathcal{S}_{j} \cup \{z_j\}, \ j = 1, 2, \cdots, k.$ This problem is to find a solution pair $(w, z), w \in R^m, \ z \in R^k$ of the system such that atleast one member of each set $\mathcal{T}_{j}$ is non basic. The formal statement of this problem is as follows:
	
	Given an $m \times k$ vertical block matrix $A$ of type $(m_1, m_2, \cdots, m_k)$ and a given vector $q \in R^m$ where $m = \sum_{j = 1}^{k} m_j,$ find $w \in R^m$ and $z \in R^k$ such that 
	\begin{center}
		$w - Az = q, \ w \geq 0, \ z \geq 0$\\
		$z_j\prod_{i = 1}^{m_j} w_{i}^{j} = 0,$ for $j = 1, 2, \cdots, k$.
	\end{center}
	This generalization is known as \textit{vertical linear complementarity problem} and this problem is denoted as VLCP$(q, A)$. In recent years, a number of applications of the vertical linear complementarity problem have been reported in the literature. Cottle-Dantzig algorithm \cite{cd} is well known for solving vertical linear complementarity problem. For pivotal algorithms see \cite{icmc}. 
	
	\subsubsection{Results in VLCP Theory}
	We give some definitions and results which will be required in the next section.
	\begin{defn}
		Let $A$ be a vertical block matrix of type $(m_1, m_2, \cdots, m_k).$ A submatrix of size $k$ of $A$ is called representative submatrix if its $j$th row is drawn from $j$th block $A^j$ of $A.$ 
	\end{defn}
	\begin{remk}
		A vertical block matrix $A$ of type $(m_1, m_2, \cdots, m_k)$ has atmost $\prod_{j=1}^{k}m_j$ distinct representative submatrices.
	\end{remk}
	\begin{defn}
		A vertical block matrix $A$ of type $(m_1, m_2, \cdots, m_k)$ is called vertical block $P(P_0)$ matrix, if all its representative submatrices are $P(P_0)$ matrices.
	\end{defn}
	
	Mohan et al. \cite{MNS} consider a vertical block matrix $A$ of type $(m_1, m_2, \cdots, m_k)$ where $m_j$ is the size of the $j$th block. Construct a matrix $M$ by copying $A_{.j}$ i.e. $j$th column of $A$, $m_j$ times for $j = 1, 2, \cdots, k$. This leads to a square matrix $M$ of order $m \times m$ from $A$. $M$ is said to be equivalent square matrix of $A$. Now we state the following theorem which will be required for our proposed algorithm.
	\begin{theorem} \cite{MNS} \label{mns}
		Given the VLCP$(q, A),$ let $M$ be the equivalent square matrix of $A.$ Then VLCP$(q, A)$ has a solution if and only if LCP$(q, M)$ has a solution.
	\end{theorem}
	\subsection{Generalized Leontief Model}
	Ebiefung and Kostreva \cite{ebiefung} extended Leontief model considering multiple technology. Further they formulated the generalized Leontief model as vertical linear complementarity problem. The formulation is given below.
	
	Consider there are total $m_j \geq 1$ number of different technologies and corresponding output with atleast one $m_j > 1$ for $j = 1, 2, \cdots, n.$ The generalized Leontief model can be written as
	\be \label{glm}
	x_j \geq b_{i}^{j} + \sum_{k = 1}^{n}a_{ik}^{j}x_k, i = 1, 2, \cdots m_j.
	\ee
	For $j = 1, 2, \cdots, n$, $b_{i}^{j}$ denotes the demand for output $i$ at industry $j.$ If $b_{i}^{j} \geq 0$ then $b_{i}^{j}$ represents quantity of goods $i$ to be produced by industry $j$ and if $b_{i}^{j} \leq 0$ then $b_{i}^{j}$ represents the number of units already available at the beginning of the period to satisfy demand.  
	
	Consider a matrix $E$ of order $m \times k$ as follows:
	\begin{center}
		$E$ = $\left[\begin{array}{rrrr}
		e^1 & 0 & \cdots & 0 \\
		0 & e^2 & \cdots & 0 \\
		\vdots &  &  & \vdots\\
		0 & 0 & \cdots & e^k
		\end{array}\right]$
	\end{center}
	where $e^j$ be a column vector of dimension $m_j \times 1$ with each component 1. Now by setting $m = \sum_{j = 1}^{n}m_j,$ we obtain $E$ of dimension $m \times n$ with $m \geq n.$ Consider
	\begin{center}
		$b^j = (b_{i}^{j})_{i = 1}^{m_j},$\\
		$A^j = (a_{ik}^{j})_{i, k = 1}^{m_j, n}$.
	\end{center}
	We write
	\begin{center}
		$b$ =  $\left[\begin{array}{c}
		b^1  \\
		b^2 \\
		\vdots\\
		b^n
		\end{array}\right],$
		\quad
		$A$ =  $\left[\begin{array}{c}
		A^1 \\
		A^2 \\
		\vdots\\
		A^n
		\end{array}\right].$
	\end{center}
	Note that $A$ is a matrix of order $m \times n.$ We write $N = E - A.$ Here $N$ is a vertical matrix of order $m \times n$ and of type $(m_1, m_2, \cdots, m_n).$ Then from the generalized Leontief model \ref{glm}, we write
	\begin{center}
		$x \geq 0$\\
		$Nx \geq b$\\
		and $x_j\prod_{i = 1}^{m_j}(N^jx - b^j)_i = 0,$ for $j = 1, 2, \cdots, n.$
	\end{center}
	Note that complementary condition states minimum cost requirement. Ebiefung \cite{ebiefung} extended the Chandrasekharan algorithm to gave an approach for solving generalized Leontief model.  
	\section{Results} 
	In this section we propose an iterative descent method based on infeasible interior point algorithm to solve a generalized Leontief model. We define the feasible region of the LCP$(q, M)$ as FEA$(q, M)$ 
	\begin{center}
		FEA$(q, M) = \{(z, w) \in R^{n} : z \geq 0,\; w \geq 0,\; w = q + Mz\}$
	\end{center} 
	and interior of the set FEA$(q, M)$ as
	\begin{center}
		FEA$_{+}(q, M) = \{(z, w) \in R^{n} : z > 0,\; w > 0\}.$
	\end{center} 
	The algorithm moves from the current iterate $(z^k, w^k)$ to the solution of the LCP$(q, M)$ by introducing $(z^{k+1}, w^{k+1})$ defined as 
	\begin{center}
		$z^{k+1} = z^k + \alpha_{k} d_{z}^{k}$\\
		$w^{k+1} = w^k + \alpha_{k} d_{w}^{k}$
	\end{center}
	\NI where $\alpha_k$ is the suitable step length of the algorithm. The generated sequence $\{(z^k, w^k)\}$ is required to satisfy $z^k > 0,\ w^k > 0.$ Now we define central trajectory as the set of solutions $(z, w) > 0$ to the system of equations
	\begin{center}
		$w = q + Mz,$\\
		$ZWe = \mu e,$
	\end{center}
	for every $\mu > 0.$ Here the neighborhood $\mathcal{N}$ of the central trajectory be defined as
	
	$$\mathcal{N} = \{(z, w) > 0 : z_iw_i \geq \gamma \frac{z^Tw}{n} (i = 1, 2, \cdots, n), z^Tw \geq \gamma^{'} \|w - Mz - q\| \ \text{or} \ \|w - Mz - q\| \leq \epsilon \}$$
	
	where $\epsilon > 0,$ $\gamma^{'} > 0$ and $\gamma \in (0, 1).$ Starting from a strictly positive point $(z^0, w^0)$ the algorithm iteratively generates a sequence $\{(z^k, w^k)\}.$ We define a non-linear system with non-negative constraints
	\be \label{newton}
	F(z, w) =  \left[\begin{array}{c}
		-Mz + w - q \\
		ZWe - \mu e 
	\end{array}\right] = 0
	\ee
	where $Z =$ diag$(z_i),$ $W =$ diag$(w_i)$ and $e$ is the vector of all 1's. We obtain
	\begin{center}
		$F^{'}(z, w)$ =  $\left[\begin{array}{rr}
		-M &  I \\
		W &  Z
		\end{array}\right].$
	\end{center}
	We apply Newton method to find the search direction $(d_{z}^{k}, \ d_{w}^{k})$ for the algorithm. Now we solve 
	\begin{center}
		$\left[\begin{array}{rr}
		-M & I \\
		W^k & Z^k
		\end{array}\right] 
		\left[\begin{array}{r}
		d_{z}^{k}\\
		d_{w}^{k}
		\end{array}\right]$ = $\left[\begin{array}{r}
		Mz^k - w^k + q \\
		-Z^kW^ke + \mu_k e
		\end{array}\right]$
	\end{center}
	where $\sigma \in [0, 1)$ and $\mu_k = \sigma \frac{{z^{k}}^{T}w^{k}}{n}.$ Hence
	\begin{center}
		$-Md_{z}^{k} + d_{w}^{k} = Mz^k - w^k + q$\\
		$W^kd_{z}^{k} + Z^{k}d_{w}^{k} = -Z^kW^ke + \mu_k  e.$	
	\end{center}
	Therefore we get
	\be \label{1}
	d_{z}^{k} = -(Z^k M + W^k)^{-1}(Z^kq + Z^kMz^k - \mu_{k} e)\
	\ee
	\be \label{2}
	d_{w}^{k} = M(z^k + d_{z}^{k}) - w^k + q.
	\ee
	Select the step length $\alpha_{k}$ suitably such that the algorithm generates strictly positive points $(z^k, w^k > 0)$ in every step. 
	
	We consider the merit function as given in \cite{simantiraki}. 
	\begin{center}
		$\|\phi(z, w)\| = \sqrt{\|w - Mz - q\|^{2} + \|ZWe\|^{2}}.$
	\end{center} 
	We show that value of the merit function at each step reduces and algorithm stops when $\|\phi(z, w)\| \leq \delta$ for some pre-determined $\delta > 0.$  Now we state our propose algorithm for solving generalized Leontief model given in (\ref{glm}).
	\subsection{Algorithm}
	\begin{description}
		\item[Step I:] Let $M$ be an equivalent square matrix of the vertical matrix $A$ using the Theorem \ref{mns} of \cite{MNS}.
		
		\item[Step II:] Let $k = 0$. Let $(z_0, w_0) > 0,$ $\delta > 0, \ \beta \in (0, 1/2], \ \gamma \in (0, 1), \ \sigma = 0.5$. Compute the value of merit function $\phi(z^k, w^k).$
		
		\item[Step III:] If $\phi(z^k, w^k) \leq \delta$, STOP and $(z^k, w^k)$ is an approximate solution to the LCP$(q, M).$ Otherwise go to Step IV.
		
		\item[Step IV:] Let $\mu_k = \sigma \frac{{z^k}^Tw^k}{n}$ and find $(d_{z}^{k}, d_{w}^{k})$ from Equation (\ref{1}) and (\ref{2}).
		
		\item[Step V:] Compute $\alpha_k$ so that 
		
		\begin{center}
			$(z^{k+1}, w^{k+1}) = (z^k, w^k) + \alpha_k(d_{z}^{k}, d_{w}^{k}) \in \mathcal{N},$\
			
			$\phi(z^{k+1}, w^{k+1}) - \phi(z^k, w^k) \leq \alpha_k \beta \nabla \phi(z^k, w^k)^T (d_{z}^{k}, d_{w}^{k}).$
		\end{center}
		\item[Step VI:] Set $k = k+1$ and go to the Step I.
	\end{description}
	
	To process generalized Leontief model, we assume the matrix $MZ + W$ is non-singular at each step of the algorithm for any diagonal matrices $Z$ and $W$ with strictly positive elements.
	
	\begin{remk}
		Note that we assume the non-singularity of the matrix $MZ + W$ to find the solution of generalized Leontief model. However in case of solving Leontief model, similar assumption is not required.
	\end{remk}
	
	Now we prove the following theorem.
	\begin{theorem}
		In a Leontief model, det$(MZ + W) \neq 0$ for any positive diagonal matrices $Z$ and $W$.
	\end{theorem}
	\begin{proof}
		It is known that to for any two strictly positive $Z$ and $W$ the following is true:
		\begin{center}
			$(MZ + W)$ is non-singular $\implies$
			$\left[\begin{array}{rr}
			-M & I \\
			W & Z
			\end{array}\right]$
		\end{center} is non-singular. Now in LCP formulation of Leontief model, $M$ is a $P$-matrix by the Theorem \ref{berman}. Again by Lemma (4.1) by Kojima et al. \cite{Kojimainterior},
		$\left[\begin{array}{rr}
		-M & I \\
		W & Z
		\end{array}\right]$ is non-singular if and only if $M \in P_0.$ Hence det$(MZ + W) \neq 0$. \end{proof}
	\subsection{Convergence Analysis}
	In this section we show that the algorithm presented in the previous section converges to the solution of generalized Leontief model. 
	We define 
	\be \label{3.4}
	\mathcal{N^*} = \{(z, w) \in \mathcal{N} : z^Tw \geq \bar{\epsilon} , \|(z, w)\| \leq \omega^* \} \ee
	
	\NI where $\bar{\epsilon}=$ min$\{\epsilon, \gamma' \epsilon\}$ and $\omega^* > 0.$ Firstly the newton direction $(d_{z}^k, d_{w}^k)$ determined by the system of equations is a continuous function such that $(z^k, w^k) \in \mathcal{N^*}.$ As the matrix 
	\begin{center}
		$\left[\begin{array}{rr}
		-M & I \\
		W^k & Z^k
		\end{array}\right]$
	\end{center} 
	is non-singular for any $(z^k, w^k) \in \mathcal{N^*}.$ Hence the Newton direction $(d_z^k, d_w^k)$ is uniformly bounded for all $(z^k, w^k)$ over the set $\mathcal{N^*}.$ Hence we can find positive constants $\eta_1,$ $\eta_2$ such that $(d_z^k, d_w^k)$ computed at every iterations satisfies
	\be \label{bd1}
	|d^k_{z_i} d^k_{w_i} - \gamma {d^k_{z}}^Td^{k}_{w}/n| \leq \eta_1 \ee
	\be	\label{bd2}
	|{d^k_{z}}^Td^{k}_{w}| \leq \eta_2.
	\ee 
	Firstly we show that the generated sequence from the algorithm $\{z^k, w^k\}$ satisfies the following condition.
	\begin{lemma}
		The generated sequence $\{z^k, w^k\}$ by the algorithm satisfies
		\begin{center}
			\begin{enumerate}[(i)]
				\item $-Mz^{k+1} + w^{k+1} -q = (1 - \alpha_k)(-Mz^k + w^k -q)$
				
				\item ${z^{k}}^{T}d_{w}^{k} + {w^{k}}^{T}d_{z}^{k} = -(1 - {\sigma}){z^{k}}^{T}w^{k}.$
			\end{enumerate}
		\end{center}
	\end{lemma}
	\begin{proof}
		$(i)$ From the Newton equation we get 
		\begin{center}
			$\left[\begin{array}{rr}
			-M & I \\
			W^k & Z^k
			\end{array}\right] 
			\left[\begin{array}{r}
			d_{z}^{k}\\
			d_{w}^{k}
			\end{array}\right]$ = $\left[\begin{array}{r}
			Mz^k - w^k + q \\
			-Z^kW^ke + \mu_k e
			\end{array}\right]$
		\end{center}
		where $\sigma \in [0, 1)$ and $\mu_k = \sigma \frac{{z^{k}}^{T}w^{k}}{n}.$ 
		Hence,
		\begin{center}
			$-Md_{z}^{k} + d_{w}^{k} = q + Mz^{k} -w^{k}$\\
			$-Md_{z}^{k} - Mz^{k} + d_{w}^{k} + w^{k} =q$\\
			$-M \alpha_k d_{z}^{k} - M \alpha_k z^k + \alpha_k d_{w}^{k} + \alpha_k w^k = \alpha_k q$\\
			$-M(z^{k+1} - z^k) - M \alpha_k z^k + (w^{k+1} - w^{k}) + \alpha_k w^{k}=\alpha_k q$\\
			$-Mz^{k+1} + w^{k+1} - q = -Mz^{k} + M \alpha_k z^{k} + w^{k} - \alpha_k w^{k} - q + \alpha_k q$\\
			$-Mz^{k+1} + w^{k+1} - q = (-Mz^{k} + w^{k} - q) - \alpha_k(-Mz^{k} + w^{k} - q)$\\
			$-Mz^{k+1} + w^{k+1} - q = (1 - \alpha_k)(-Mz^{k} + w^{k} - q).$
		\end{center} 
		
		$(ii)$ 2nd part follows form the Newton equation by restricting $\mu_k = \sigma {z^{k}}^{T} w^{k}.$ We have
		\begin{center}
			$W^kd_{z}^{k} + Z^{k}d_{w}^{k} = - Z^kW^ke + \mu_k e $\\
			${z^{k}}^{T}d_{w}^{k} + {w^{k}}^{T}d_{z}^{k} = - {z^{k}}^{T}w^{k} + \sigma {{z^k}^Tw^k}$\\
			${z^{k}}^{T}d_{w}^{k} + {w^{k}}^{T}d_{z}^{k} = -(1 - \sigma) {z^{k}}^{T}w^{k}.$
		\end{center}
	\end{proof}
	To specify the selection of suitable step length $\alpha_{k},$ we define real valued functions $g^{I}(\alpha)$ and $g^{II}(\alpha)$ as follows: 
	\begin{center}
		$g^{I}_{i}(\alpha) = z_{i}^{k+1}w_{i}^{k+1} - \gamma \frac{{z^{k+1}}^{T}w^{k+1}}{n}$
	\end{center}
	\begin{center}	
		$g^{II}(\alpha) = {z^{k+1}}^{T}w^{k+1} - \gamma \|w^k - Mz^k -q\|$
	\end{center}
	where $\gamma \in (0, 1).$ The next lemma will be required for the choice of $\alpha_k.$ The proof of the following two lemmas are in line with the idea given in \cite{Kojimainterior}. 
	\begin{lemma} \label{lemma 1}
		If $g^{I}_{i}(\alpha) = z_{i}^{k+1}w_{i}^{k+1} - \gamma \frac{{z^{k+1}}^{T}w^{k+1}}{n},$ then
		$g^{I}_{i}(\alpha) \geq \sigma(1- \gamma)(\bar{\epsilon}/n)\alpha - \eta_{1} \alpha^{2}$ $\forall$ $i.$
	\end{lemma}
	\begin{proof}
		$$\begin{array}{ll}
		g^{I}_{i}(\alpha) &= z_{i}^{k+1}w_{i}^{k+1} - \gamma \frac{{z^{k+1}}^{T}w^{k+1}}{n} \\
		&=z_{i}^{k}w_{i}^k + \alpha (z_i^k d_{w_{i}}^{k} + w_i^k d_{z_{i}}^{k}) + \alpha^{2} d^{k}_{z_{i}} d^{k}_{w_{i}} - \gamma \frac{{z^{k}}^T w^{k} + \alpha ({z^{k}}^{T} d^{k}_{w} + {w^{k}}^{T} d^{k}_{z}) + \alpha^{2} {d^{k}_{z}}^T d^{k}_{w}  } {n}\\
		&=z_{i}^{k}w_{i}^k - \alpha(z_{i}^{k}w_{i}^{k} - \sigma{z^k}^{T}w^k/n)+ \alpha^{2}d_{z_{i}}^{k}d_{w_{i}}^{k} - \gamma\frac{{z^k}^{T}w^k - \alpha(1 - \sigma){z^k}^{T}w^{k} + \alpha^{2}{d^{k}_{z}}^T d^{k}_{w}}{n}\\
		&=(1 - \alpha)z_{i}^{k}w_{i}^k + \alpha \sigma{z^k}^{T}w^k/n + \alpha^{2}d_{z_{i}}^{k}d_{w_{i}}^{k} - \gamma\frac{(1- \alpha){z^k}^{T}w^{k}+ \alpha\sigma{z^k}^{T}w^{k}+ \alpha^{2}{d^{k}_{z}}^T d^{k}_{w}}{n}  \\
		&=(1-\alpha)(z_{i}^{k}w_{i}^k - \gamma{z^k}^{T}w^{k}/n)+ \sigma (1 - \gamma)\alpha{z^k}^{T}w^k/n + \alpha^{2}(d_{z_{i}}^{k}d_{w_{i}}^{k} - \gamma {d_{z}^{k}}^Td_{w}^{k}/n)\\
		&\geq \sigma(1- \gamma)(\bar{\epsilon}/n)\alpha - \eta_{1} \alpha^{2} (\text{by} \, (\ref{3.4}), \, (\ref{bd1})).
		\end{array}$$
	\end{proof}
	\begin{lemma} \label{lemma 2}
		If $g^{II}(\alpha) = {z^{k+1}}^{T}w^{k+1} - \gamma \|w^{k+1} - Mz^{k+1} -q\|,$ then
		$g^{II}(\alpha) \geq \mu_{k}\alpha - \alpha^{2}\eta_{2}.$
	\end{lemma}
	\begin{proof}
		$$ \begin{array}{ll}
		g^{II}(\alpha) &= {z^{k+1}}^{T}w^{k+1} - \gamma \|w^{k+1} - Mz^{k+1} -q\|\\
		&=(z^{k} + \alpha d_{z}^{k})^T(w^k + \alpha d_{w}^{k}) - \gamma(1- \alpha)\|w^k - Mz^k -q\|\\
		&= {z^{k}}^{T}w^k + \alpha [-(1 - \sigma/n)]{z^{k}}^{T}w^k - \gamma(1- \alpha)\|w^k - Mz^k -q\| + \alpha^2 {d_{z}^{k}}^Td_{w}^{k}\\
		&=(1- \alpha)({z^k}^{T}w^k - \gamma \|w^k - Mz^k -q\|) + \mu_{k}\alpha + \alpha^{2}{d^{k}_{z}}^T d^{k}_{w}\\
		&\geq \mu_{k}\alpha - \alpha^{2}\|{d_{z}^{k}}^Td_{w}^{k}\|\\
		&\geq \mu_{k}\alpha - \alpha^{2}\eta_{2} (\text{by} \, (\ref{bd2})).
		\end{array}$$
	\end{proof}
	
	To find step length we have to find $\alpha_{k}$ to be the largest number $\in (0, 1]$ for which $g^{I}_{i}(\alpha) \geq 0$  and $g^{II}(\alpha) \geq 0.$ From Lemma (\ref{lemma 1}) and Lemma (\ref{lemma 2}) we have $g^{I}_{i}(\alpha) \geq 0, \ g^{II}(\alpha) \geq 0.$ Hence we can easily compute the lower bound of $\alpha_{k}$ by solving them for $\alpha.$ Hence letting $\alpha_k=$ $\min$$\{1, \sigma \bar{\epsilon}(1 - \gamma)/n\eta_1, \mu_k/ \eta_2 \}$, we obtain
	\begin{center}
		$g^{I}_{i}(\alpha) \geq 0$\\
		$g^{II}(\alpha) \geq 0$
	\end{center}
	hold for every $\alpha \in [0, \alpha_k]$ and $\forall i.$
	
	In \cite{simantiraki}, Simantiraki and Shanno state the idea to show the descent directions of the proposed algorithm. We prove the following theorem to show that the directions $(d_{z}^{k}, d_{w}^{k})$ generated by the algorithm is descent direction.
	\begin{theorem}
		The directions $(d_{z}^{k}, d_{w}^{k})$ generated at the $k$th iteration by the algorithm is a descent direction for the merit function $\phi(z^k, w^k).$  
	\end{theorem}	
	\begin{proof}
		Consider $p^k = (z^k, w^k),$ then
		\begin{center}
			$d_{p}^k = \mbox{F}'(p^k)[-F(p^k) + \mu_k \bar{e}]$\\
			$\nabla\phi(p^k) = \frac{1}{\|F(p^k)\|} \mbox{F}'(p^k)^T F(p^k),$
		\end{center}
		where $\bar{e}$ is the vector $(0, e)^T.$ Hence
		$$\begin{array}{ll}
		\nabla \phi(p^k) d_{p}^{k} &= \frac{1}{\|\mbox{F}(p^k)\|} \mbox{F}(p^k)^T \mbox{F}^\prime(p^k)\mbox{F}'(p^k)^{-1}[-F(p^k) + \mu_k \bar{e}]\\
		&=\frac{1}{\|\mbox{F}(p^k)\|}(-\|\mbox{F}(p^k)\|^2 + \mu_k{z^{k}}^Tw^k)\\
		&=-(\|\mbox{F}(p^k)\| - \mu_k\frac{{z^k}^Tw^k}{\|\mbox{F}(p^k)\|})\\
		&=-(\phi(p^k) - \mu_k\frac{{z^k}^Tw^k}{\phi(p^k)})
		\end{array}$$ 
		Again we have $\mu_k = \sigma \frac{{z^k}^Tw^k}{n}$
		$$\begin{array}{ll}
		\mu_k{z^k}^Tw^k &= \sigma\frac{({z^k}^Tw^k)^2}{n}\\
		&\leq \sigma\|Z^kW^ke\|^2\\
		&\leq \sigma\phi(p^k)^2 
		\end{array}$$ 
		Hence
		$$\begin{array}{ll}
		\nabla\phi(p^k)^Td^{k}_{p}&\leq -\phi(p^k)(1-\sigma)\\
		&\leq 0.
		\end{array}$$ 
		Hence $(d_{z}^{k}, d_{w}^{k})$ is the decent direction for the algorithm. 
	\end{proof}
	Now we prove the convergence result of the proposed algorithm.
	\begin{theorem}
		Let the sequence $p^k = \{z^k, w^k\}$ be generated by the algorithm. Then for any $\{\frac{n \mu_k}{{z^k}^Tw^k}\} \subset (0, 1)$ bounded away from zero, $\phi(p^k)$ converges to zero.  
	\end{theorem}
	\begin{proof} We have
		$$\begin{array}{ll}
		\phi(p^{k+1}) &= \alpha_{k} \beta \nabla \phi(p^k)^{T}d_p^k\\
		& \leq \phi(p^k) - \alpha_k \beta (1- \sigma)\phi(p^k)\\
		& \leq [1 - \alpha_k \beta (1 - \sigma)]\phi(p^k).
		\end{array}$$ 
		Hence,
		\begin{center}
			$\frac{\phi(p^{k+1})}{\phi(p^k)} \leq 1 - \alpha_k \beta (1 - \sigma).$
		\end{center}
		Now $\alpha_k'$s are bounded away from zero. Hence corresponding subsequences converges to zero. So letting $k \raro \infty,$ we get
		\begin{center}
			$\frac{\phi(p^{k+1})}{\phi(p^k)} < 1.$
		\end{center}
		Hence $\{\phi(p^{k})\}$ converges to zero.
	\end{proof}
	\section{Numerical illustration}
	We consider the example given by Ebiefung et al. \cite{ebiefung} where they consider an economy with three sectors. The outputs of the sectors are pair of shoes, food and light bulbs. The technology matrix and demands for the economy are
	\begin{center}
		$A$ =  $\left[\begin{array}{rrr}
		0.6 & 0.1 & 0.3 \\
		0.3 & 0.6 & 0.1 \\
		0.1 & 0.3 & 0.6
		\end{array}\right],$
		\quad
		$h^I$ =  $\left[\begin{array}{r}
		150 \\
		-500\\
		-20
		\end{array}\right].$	
	\end{center}
	Here the columns and rows are given in the order of shoes, food and light bulbs in that order. Consider a new technology for each of these sectors has been introduced in the market. The new technology matrix $B$ and demand vector $h^{II}$ are
	\begin{center}
		$A$ =  $\left[\begin{array}{rrr}
		0.5 & 0.2 & 0.3 \\
		0.4 & 0.2 & 0.4 \\
		0.1 & 0.6 & 0.3
		\end{array}\right],$
		\quad
		$h^{II}$ =  $\left[\begin{array}{r}
		150 \\
		-500\\
		-20
		\end{array}\right].$	
	\end{center}
	Now formulating this problem under as generalized Leontief input-output model, we can get a vertical block matrix $N$ of type $(2, 2, 2)$ and the demand vector $b,$
	\begin{center}
		$A$ =  $\left[\begin{array}{rrr}
		0.4 & -0.1 & -0.3 \\
		0.5 & -0.2 & -0.3 \\
		-0.3 & 0.4 & -0.1 \\
		-0.4 & 0.8 & -0.4\\
		-0.1 & -0.3 & 0.4\\
		-0.1& -0.6 & 0.7
		\end{array}\right],$
		\quad
		$b$ =  $\left[\begin{array}{r}
		150 \\
		150\\
		-500\\
		-500\\
		-20\\
		-20
		\end{array}\right].$	
	\end{center}
	
	\begin{table}[!]
		\centering
		\resizebox{\columnwidth}{!}{%
			\begin{tabular}{|c|c|c|c|c|c|c|}
				\hline
				Iteration (k) & $z^k$ & $w^k$ &  $d_z^k$ & $d_w^k$ & $\mu$ & $\|\phi(z^k, w^k)\|$\\
				\hline
				
				1 & $\left( \begin{array}{rrr}  35.925 \\ 30.340 \\ -2.034 \\ 12.449 \\ 24.486 \\ 15.5888 \end{array} \right)$ & $\left( \begin{array}{rrr} 23.746 \\ 29.331 \\ 61.705 \\ 47.222 \\ 35.184 \\ 44.083 \end{array} \right)$& $\left( \begin{array}{rrr} 54.151 \\ 3.108 \\ -292.765 \\ -160.400 \\ -50.386 \\ -131.707 \end{array} \right)$ & $\left( \begin{array}{rrr} -57.151 \\ -6.108 \\ -289.765 \\ 157.400 \\ 47.386 \\ 128.707 \end{array} \right)$ & 810 & 2316.635\\
				\hline

				2 & $\left( \begin{array}{rrr}  59.865 \\ 41.275 \\ 2.604 \\ -4.277 \\ 24.253 \\ 14.622 \end{array} \right)$ & $\left( \begin{array}{rrr} 6.260 \\ 16.541 \\ 132.708 \\ 110.261 \\ 33.010 \\ 45.175 \end{array} \right)$& $\left( \begin{array}{rrr} 116.275 \\ 53.114 \\ 22.527 \\ -81.240 \\ -1.132 \\ -4.690 \end{array} \right)$ & $\left( \begin{array}{rrr} -84.930 \\ -62.119 \\ 344.853 \\ 306.176 \\ -10.560 \\ 5.306 \end{array} \right)$ & 563.12 & 1870.186\\
				
				\hline

				3 & $\left( \begin{array}{rrr}  154.003 \\ 37.663 \\ 0.874 \\ 1.495 \\ 27.613 \\ 15.783 \end{array} \right)$ & $\left( \begin{array}{rrr} -3.676 \\ 15.253 \\ 222.613 \\ 191.375 \\ 22.087 \\ 34.395 \end{array} \right)$& $\left( \begin{array}{rrr} 269.986 \\ -10.359 \\ -4.959 \\ 16.557 \\ 9.636 \\ 3.329 \end{array} \right)$ & $\left( \begin{array}{rrr} -28.497 \\ -3.694 \\ 257.846 \\ 232.633 \\ -31.328 \\ -30.917 \end{array} \right)$ & 358.90 & 1509.764\\
				\hline
				
				4 & $\left( \begin{array}{rrr}  283.651 \\ -3.8924 \\ 0.765 \\ 1.126 \\ 32.314 \\ 15.883 \end{array} \right)$ & $\left( \begin{array}{rrr} 1.258 \\ 29.044 \\ 271.001 \\ 229.323 \\ 13.735 \\ 27.627 \end{array} \right)$& $\left( \begin{array}{rrr} 371.827 \\ -119.182 \\ -0.313  \\ -1.0591 \\ 13.4831 \\ 0.2845 \end{array} \right)$ & $\left( \begin{array}{rrr} 14.152 \\ 39.553 \\ 138.775 \\ 108.831 \\ -23.951 \\ -19.409 \end{array} \right)$ & 246.31 & 1242.866 \\
				\hline

				5 & $\left( \begin{array}{rrr}  410.891 \\ 9.991 \\ 0.606 \\ 0.651 \\ 46.250 \\ 13.582 \end{array} \right)$ & $\left( \begin{array}{rrr} 0.277 \\ 42.240 \\ 368.255 \\ 308.720 \\ 1.467 \\ 19.039 \end{array} \right)$& $\left( \begin{array}{rrr} 127.239 \\ 13.883 \\ -0.158  \\ -0.474 \\ 13.935 \\ -2.301 \end{array} \right)$ & $\left( \begin{array}{rrr} -0.980 \\ 13.195 \\ 97.253 \\ 79.397 \\ -12.268 \\ -8.587 \end{array} \right)$ & 238.837 & 827.699 \\
				\hline
				
				\vdots  & \vdots & \vdots & \vdots & \vdots & \vdots & \vdots \\
				\hline

				49 & $\left( \begin{array}{rrr}  415.45 \\ 0.0504 \\ 0.0056 \\ 0.0067 \\ 53.8 \\ 0.129 \end{array} \right)$ & $\left( \begin{array}{rrr} 0.00504 \\ 41.5 \\ 369.95 \\ 312.21 \\ 0.0388 \\ 16.22 \end{array} \right)$& $\left( \begin{array}{rrr} -0.0074 \\ -0.0055 \\ -0.00062  \\ -0.00074  \\ -0.00083 \\ -0.0142 \end{array} \right)$ & $\left( \begin{array}{rrr} -0.00056 \\ -0.00172 \\ 0.00488 \\ 0.01016 \\ -0.00432 \\ -0.00843 \end{array} \right)$ & 1.8853 & 5.1311 \\
				\hline
				
			\end{tabular}%
		}
	\end{table}
	
	\begin{table}[h]
		\centering
		\resizebox{\columnwidth}{!}{%
			\begin{tabular}{|c|c|c|c|c|c|c|}
				\hline
				Iteration (k) & $z^k$ & $w^k$ &  $d_z^k$ & $d_w^k$ & $\mu$ & $\|\phi(z^k, w^k)\|$\\
				\hline
				
				50 & $\left( \begin{array}{rrr}  415.44 \\ 0.0453 \\ 0.00509 \\ 0.00603 \\ 53.85 \\ 0.116 \end{array} \right)$ & $\left( \begin{array}{rrr} 0.00453 \\ 41.55 \\ 369.96 \\ 312.22 \\ 0.0350 \\ 16.222 \end{array} \right)$& $\left( \begin{array}{rrr} -0.0067 \\ -0.0050 \\ -0.0005  \\ -0.00067  \\ -0.0007 \\ -0.0128 \end{array} \right)$ & $\left( \begin{array}{rrr} -0.0005 \\ -0.00155 \\ 0.00439 \\ 0.00915 \\ -0.00388 \\ -0.00759 \end{array} \right)$ & 1.8853 & 5.1311 \\
				\hline

				\vdots  & \vdots & \vdots & \vdots & \vdots & \vdots & \vdots \\
				
				\hline
				
				99 & $\left( \begin{array}{rrr}  415.38 \\ 0.000001 \\ 0.000000 \\ 0.000000 \\ 53.84 \\ 0.000003 \end{array} \right)$ & $\left( \begin{array}{rrr} 0.000000 \\ 41.53 \\ 370 \\ 312.3 \\ 0.000001 \\ 16.15 \end{array} \right)$& $\left( \begin{array}{rrr} -0.000000 \\ -0.000000 \\ -0.000000  \\ -0.000000  \\ -0.000000 \\ -0.000000 \end{array} \right)$ & $\left( \begin{array}{rrr} -0.000000 \\ -0.000000 \\ 0.000000 \\ 0.000000 \\ -0.000000 \\ -0.000000 \end{array} \right)$ & 0.000055 & 0.000151 \\
				\hline
				
				
				
			\end{tabular}%
		}
		\caption{Summary of computation for the proposed algorithm} \label{t1}
	\end{table}
	\pagebreak
	The algorithm runs on a HP PC with intel Core i5
	processor 3.10 GHz 4 GB of RAM. The proposed algorithm converges to the solution $[415, 0, 54]^T$ after $99$ iterations and time taken to reach the optimal solution is $1.12 \times 10^{-1}$s. 
	
	\NI{\bf Acknowledgment:} The author R. Jana is thankful to the Department of Science and Technology, Govt. of India, INSPIRE Fellowship Scheme for financial support. 
	
	\bibliographystyle{plain}
	\bibliography{bibfile}

\begin{thebibliography}{10}

\bibitem{Barker}
Terence~S Barker.
\newblock {\em Foreign trade in multisectoral models}.
\newblock University of Cambridge, Department of Applied Economics, 1973.

\bibitem{berman}
Abraham Berman and Robert~J Plemmons.
\newblock {\em Nonnegative matrices in the mathematical sciences}, volume~9.
\newblock Siam, 1994.

\bibitem{cd}
Richard~W Cottle and George~B Dantzig.
\newblock A generalization of the linear complementarity problem.
\newblock {\em Journal of Combinatorial Theory}, 8(1):79--90, 1970.

\bibitem{Dantzig}
George~B Dantzig.
\newblock Optimal solution of a dynamic leontief model with substitution.
\newblock {\em Econometrica: Journal of the Econometric Society}, pages
  295--302, 1955.

\bibitem{icmc}
AK~Das, R~Jana, et~al.
\newblock Finiteness of criss-cross method in complementarity problem.
\newblock In {\em International Conference on Mathematics and Computing}, pages
  170--180. Springer, 2017.

\bibitem{ebiefung}
Aniekan~A Ebiefung and Michael~M Kostreva.
\newblock The generalized leontief input-output model and its application to
  the choice of new technology.
\newblock {\em Annals of Operations Research}, 44(2):161--172, 1993.

\bibitem{kojima}
Masakazu Kojima, Nimrod Megiddo, and Shinji Mizuno.
\newblock A primal—dual infeasible-interior-point algorithm for linear
  programming.
\newblock {\em Mathematical programming}, 61(1-3):263--280, 1993.

\bibitem{Kojimainterior}
Masakazu Kojima, Nimrod Megiddo, Toshihito Noma, and Akiko Yoshise.
\newblock A unified approach to interior point algorithms for linear
  complementarity problems: A summary.
\newblock {\em Operations Research Letters}, 10(5):247--254, 1991.

\bibitem{leontief1}
Wassily Leontief.
\newblock {\em Input-output economics}.
\newblock Oxford University Press, 1986.

\bibitem{MNS}
SR~Mohan, SK~Neogy, and R~Sridhar.
\newblock The generalized linear complementarity problem revisited.
\newblock {\em Mathematical Programming}, 74(2):197, 1996.

\bibitem{variants}
Jan Oosterhaven.
\newblock On the plausibility of the supply-driven input-output model.
\newblock {\em Journal of Regional Science}, 28(2):203--217, 1988.

\bibitem{simantiraki}
Evangelia~M Simantiraki and David~F Shanno.
\newblock An infeasible-interior-point method for linear complementarity
  problems.
\newblock {\em SIAM Journal on Optimization}, 7(3):620--640, 1997.

\bibitem{veinott}
Arthur~F Veinott~Jr.
\newblock Minimum concave-cost solution of leontief substitution models of
  multi-facility inventory systems.
\newblock {\em Operations Research}, 17(2):262--291, 1969.

\bibitem{regionaleco}
Thomas Wiedmann.
\newblock A review of recent multi-region input--output models used for
  consumption-based emission and resource accounting.
\newblock {\em Ecological Economics}, 69(2):211--222, 2009.

\end{thebibliography}
\end{document}